\documentclass[12pt]{article} 
\usepackage{amsmath,amsthm}
\usepackage{amsfonts,amsmath}
\usepackage{amssymb}

\usepackage{color} 
\usepackage{colortbl} 
 \usepackage{epsfig}
  \usepackage{graphics}
  \usepackage{amssymb}
  \usepackage{color}

\newtheorem{theorem}{Theorem}[section]  \newtheorem{corollary}[theorem]{Corollary}
 \newtheorem{example}[theorem]{Example} 

\renewcommand{\geq}{\geqslant}
\renewcommand{\leq}{\leqslant}
\renewcommand{\ge}{\geqslant}
\renewcommand{\le}{\leqslant}

\def\cref#1{Corollary~$\ref{#1}$}

\usepackage{rotating} \input epsf

\begin{document}

\title{Embedding partial Latin squares in Latin squares with many mutually orthogonal mates }

\author{Diane M. Donovan
 \\
School of Mathematics and Physics\\ University of Queensland, Brisbane 4072 Australia \\ \texttt{(dmd@maths.uq.edu.au)}\\ \\
Mike Grannell\\ School of Mathematics and Statistics \\ The Open
University \\ Walton Hall, Milton Keynes MK7 6AA\\ United Kingdom\\ \texttt{(mike.grannell@open.ac.uk)}\\ \\ Emine
\c{S}ule Yaz{\i}c{\i}\footnote{This work was supported by Scientific and
Technological Research Council of Turkey TUBITAK Grant Number: 116F166}\\ Department of Mathematics,\\
 Ko\c{c} University,\\
  Sar{\i}yer,
34450,\\
 \.{I}stanbul, Turkey\\
\texttt{(eyazici@ku.edu.tr)}\\ }

\maketitle

\begin{abstract}

We show that any partial Latin square of order $n$ can be embedded in a Latin square of order at most $16n^2$ which has
at least $2n$ mutually orthogonal mates. We also show that for any
$t\geq 2$, a pair of orthogonal partial Latin squares of order $n$ can be embedded into a set of $t$ mutually
orthogonal Latin squares (MOLS) of order a polynomial with respect to $n$.
Furthermore, the constructions that we provide show that MOLS($n^2$)$\geq$MOLS($n$)+2, consequently we  give a set of $9$ MOLS($576$). The maximum known size of a set of
MOLS($576$) was previously given as $8$ in the literature. \end{abstract}

\section{Introduction}

In 1960 Evans \cite{Evans} showed that it was possible to embed any partial Latin square of order $n$ in some Latin
square of order $t$, for every $t\geq 2n$, where $2n$ is a tight bound. In
the same paper Evans  raised the question of embedding orthogonal partial Latin squares in sets of mutually orthogonal
Latin squares.

The importance and relevance of this question is demonstrated by the prevalence and application of orthogonal Latin
squares  to other areas of mathematics (see \cite{Handbook}). For instance,
the existence of a set of $n-1$ mutually orthogonal Latin squares of order $n$ is equivalent to the existence of a
projective plane of order $n$ (see \cite{Design} for a relevant
construction). Thus results on the embedding of orthogonal partial Latin squares provide information on the embedding
of sets of partial lines in finite geometries. In addition, early
embedding results for partial Steiner triple systems utilised embeddings of partial idempotent Latin squares (see for
example \cite{6n+3}). It has also been suggested that embeddings of block
designs with block size $4$ and embeddings of Kirkman triple systems may  make use of embeddings of pairs of
orthogonal partial Latin squares (see \cite{HRW}).

In 1976 Lindner \cite{Lindner}  showed that a pair of orthogonal partial Latin squares can always be finitely embedded
in a pair of orthogonal Latin squares. However, there was no known
method for obtaining an embedding of polynomial order (with respect to the order of the partial arrays). In \cite{HRW},
Hilton \emph{et al.} formulate some necessary conditions for a pair of
orthogonal partial Latin squares to be embedded in a pair of orthogonal Latin squares. Then in \cite{Jenkins} Jenkins
developed a construction for   embedding a single partial Latin square
of order $n$ in a Latin square of order $4n^2$ for which there exists an orthogonal mate. In  2014, Donovan and
Yaz{\i}c{\i} developed a construction that verified that a pair of orthogonal partial
Latin squares, of order $n$, can be embedded in a pair of orthogonal Latin squares of order at most $16n^4$.

This paper seeks to extend these results, providing new constructions that show that a partial Latin square, of order
$n$, can be embedded in a Latin square, of order at most $16n^2$ with
many mutually orthogonal mates. Further, we develop a second construction for embedding a pair of orthogonal partial
Latin squares of order $n$ in sets of mutually orthogonal Latin squares of
any size where the Latin squares have polynomial order with respect to $n$. Also, as a corollary, the construction can
be used to increase the best known lower bound for the largest set of
MOLS($576$). In the literature the existence of $8$ MOLS($576$) is established. However, we construct $9$ MOLS($576$).

We preface the discussion of our main result with some necessary definitions.

\section{Definitions}

Let  $N=\{\alpha_1,\alpha_2,\ldots,\alpha_n\}$ represent a set of $n$ distinct elements. A non-empty subset $P$ of
$N\times N\times N$  is said to be a {\em  partial Latin square} (PLS($n$)), of order $n$, if for all
$(x_1,x_2,x_3),(y_1,y_2,y_3)\in P$ and for all distinct $i,j,k\in \{1,2,3\}$,
\begin{align*} x_i=y_i\mbox{ and }x_j=y_j\mbox{ implies }x_k=y_k.
\end{align*}
We say that $P$ is \emph{indexed} by $N$. We may think of $P$ as an $n\times n$ array where symbol $e\in N$ occurs in
cell $(r,c)$, whenever $(r,c,e)\in P$, and we will write $e=P(r,c)$.  We say that cell $(r,c)$ is {\em empty} in
$P$ if, for all  $e\in N$, $(r,c,e)\notin P$.  The {\em volume} of $P$ is $|P|$. If $|P|=n^2$, then we say that $P$ is
a {\em Latin square} (LS($n$)), of order $n$. If for all $1\leq i\leq n$, $(\alpha_i,\alpha_i,\alpha_i)\in P$, then $P$
is said to be idempotent. The set of elements $\{(x_1,x_2,x_3)\in P \mid x_1=x_2\}$ forms the main diagonal of $P$.

Two partial Latin squares $P$ and $Q$, of the same order $n$ are said to be {\em orthogonal}, denoted OPLS($n$), if
they have the same non-empty cells and for all $r_1,c_1,r_2,c_2,x,y\in N$
\begin{align*} \{(r_1,c_1,x),(r_2,c_2,x)\}\subseteq  P\mbox{ implies }\{(r_1,c_1,y),(r_2,c_2,y)\}\not\subseteq Q.
\end{align*}

\begin{example} \end{example} \begin{figure}[h!]

\begin{align*} \begin{array}{|c|c|c|c|} \hline
  0 & 1 & 2 &  \\ \hline
  2 & 0 & 1 & 3 \\ \hline
  3 &  & 0 &  \\ \hline
   & 2 &  & 1 \\ \hline
\end{array} &&\begin{array}{|c|c|c|c|} \hline
  0 & 2 & 1 &  \\ \hline
  3 & 1 & 0 & 2 \\ \hline
  1 &  & 2 &  \\ \hline
   & 0 &  & 3 \\ \hline
\end{array} \end{align*}
 \caption{A pair of orthogonal partial Latin squares of order 4}
\end{figure}

This definition extends in the obvious way to a pair {\em orthogonal Latin squares} of order $n$. A set of $t$ Latin
squares, of order $n$, which are pairwise orthogonal are said to be a set
of $t$ {\em mutually orthogonal Latin squares}, denoted  MOLS($n$).

A set $T\subseteq A$, where $A$ is a Latin square of order $n$, is said to be a {\em transversal}, if \begin{itemize}
\item $|T|=n$, and \item for all distinct $(r_1,c_1,x_1),(r_2,c_2,x_2)
\in T$, $r_1\neq r_2$, $c_1\neq c_2$ and $x_1\neq x_2$. \end{itemize}

Note that a Latin square has an orthogonal mate if and only if it can be partitioned into disjoint transversals.

We say that a partial Latin square $P$ on the set $N$ can be {\em embedded} in a Latin square $L$ on the set $M$ if
there exists one-to-one mappings $f_1,f_2,f_3:N\rightarrow M$ such that if $(x_1,x_2,x_3)\in P$ then
$(f_1(x_1),f_2(x_2),f_3(x_3))\in L$. A pair of orthogonal partial Latin squares $(P_1,P_2)$ is said to be embedded in a
pair
of orthogonal Latin squares $(L_1,L_2)$ if $P_1$ is embedded in $L_1$ and $P_2$ is embedded in $L_2$. A set of
orthogonal partial Latin squares $(P_1,P_2,\dots,P_n)$ is embedded in a set of mutually
orthogonal Latin squares $\{L_1,L_2,\dots,L_m\}$ where $m\geq n$ if $P_i$ is embedded in $L_i$ for all $1\leq i\leq n$.

This paper will make extensive use of Evans' embedding result, which is stated as: \begin{theorem}[\cite{Evans}]
\label{Evans} A partial Latin square of order $n$ can be embedded in a Latin
square of order $t$, for any $t \geq 2n$. \end{theorem}

The following is a similar embedding result for partial idempotent Latin squares.
\begin{theorem}[\cite{IdeEmbed}]
\label{IdeEm} A partial idempotent Latin square of order $n$ can be embedded in a idempotent Latin
square of order $t$, for any $t \geq 2n+1$. \end{theorem}

It is also worth noting the following well known result which is the culmination of results from a series of papers by
many authors, see for example \cite{HZ}. \begin{theorem}\label{3n} A
pair of orthogonal Latin squares of order $n$ can be embedded in a pair of orthogonal Latin squares of order $t$ if $t
\geq 3n$, with the bound of $3n$ being best possible. \end{theorem}

\section{Embedding a PLS in a set of MOLS} We begin by assuming that there exists a set of $t$ MOLS($n$) and show that
any Latin square $L$, of order $n$, can be embedded in a Latin square
${\cal B}$, of order $n^2$, with the additional property that ${\cal B}$ has $t$ mutually orthogonal mates. This result
will then allow us to show that any PLS($s$) where $s\leq n/2$ can be
embedding in a Latin square ${\cal B}$ of order $n^2$ such that ${\cal B}$ has $t$ mutually orthogonal mates. Thus
this result, and the associated construction, allows us to generalise Jenkin's
result which is stated as:

\begin{theorem}[\cite{Jenkins}] Let $L$ be a Latin square of order $n$ with $n \geq 3$ and $n \neq 6$. Then $L$ can be
embedded in a Latin square of order $n^2$ which has an orthogonal mate.
\end{theorem}

\begin{theorem}\label{one-to-many} Let $F_1=[F_1(r,c)],\dots, F_t=[F_t(r,c)]$ be $t$ mutually orthogonal Latin squares
of order $n$ indexed by $[n]=\{0,1,\ldots,n-1\}$. Let $L=[L(r,c)]$ be a Latin square of order $n$, also indexed by
$[n]$. Then the arrays ${\cal B}$ and ${\cal X}_k,$ for $1\leq k\leq t$,  form a set of $t+1$ mutually orthogonal Latin
squares of order $n^2$ where
\begin{align*} {\cal X}_k&=\{((p,r),(q,c),(F_k(F_1(p,r),q),F_k(F_1(p,q),c)))\mid 0\leq p,q,r,c\leq n-1\},\\ {\cal
B}&=\{((p,r),(q,c),(F_1(p,q),L(F_1(p,r),c)))\mid 0\leq p,q,r,c\leq n-1\}.
\end{align*}
\end{theorem}

\begin{proof} For completeness we begin by showing these arrays are Latin squares, then that ${\cal X}_k$, $1\leq k\leq
t$, are mutually orthogonal and finally that for each $k$, ${\cal X}_k$
and ${\cal B}$ are orthogonal.

Assume that one of ${\cal X}_k$, $1\leq k\leq t$ or ${\cal B}$ is not a Latin square. Then \begin{itemize} \item for
some $(p,r)$, there exists $(q,c)$ and $(q^\prime,c^\prime)$, with $(q,c)\ne (q^\prime,c^\prime)$, such that \\
$(F_k(F_1(p,r),q),F_k(F_1(p,q),c))=(F_k(F_1(p,r),q^\prime),F_k(F_1(p,q^\prime),c^\prime))$, or \\
$(F_1(p,q),L(F_1(p,r),c))=(F_1(p,q^\prime),L(F_1(p,r),c^\prime))$; \end{itemize}
     or
\begin{itemize} \item for some $(q,c)$, there exists $(p,r)$ and $(p^\prime,r^\prime)$, with $(p,r)\ne
(p^\prime,r^\prime)$, such that \\
$(F_k(F_1(p,r),q),F_k(F_1(p,q),c))=(F_k(F_1(p^\prime,r^\prime),q),F_k(F_1(p^\prime,q),c))$, or \\
$(F_1(p,q),L(F_1(p,r),c))=(F_1(p^\prime,q),L(F_1(p^\prime,r^\prime),c))$. \end{itemize}

The first case  implies \begin{align*} F_k(F_1(p,r),q)=F_k(F_1(p,r),q^\prime)\mbox{ and }
F_k(F_1(p,q),c)=F_k(F_1(p,q^\prime),c^\prime). \end{align*} Thus we may deduce that $q=q^\prime $ and
consequently $c=c^\prime$, a contradiction. All the other cases follow in a similar manner and hence  ${\cal X}_k$,
$1\leq k\leq t$, and ${\cal B}$ are Latin squares of order $n^2$.

Next assume that ${\cal X}_k$ and ${\cal X}_{\ell}$, for $k\neq \ell$ are not orthogonal, and so there exist distinct
cells $((p,r), (q,c))$ and $((p^\prime,r^\prime), (q^\prime,c^\prime))$
such that \begin{align*}
(F_k(F_1(p,r),q),F_k(F_1(p,q),c))&=(F_k(F_1(p^\prime,r^\prime),q^\prime),F_k(F_1(p^\prime,q^\prime),c^\prime))\mbox{
and}\\
(F_{\ell}(F_1(p,r),q),F_{\ell}(F_1(p,q),c))&=(F_\ell(F_1(p^\prime,r^\prime),q^\prime),F_{\ell}(F_1(p^\prime,q^\prime),
c^\prime)). \end{align*} Then \begin{align}
F_k(F_1(p,r),q)&=F_k(F_1(p^\prime,r^\prime),q^\prime),\label{gen-eq1}\\
F_k(F_1(p,q),c)&=F_k(F_1(p^\prime,q^\prime),c^\prime),\label{gen-eq2}\\
F_\ell(F_1(p,r),q)&=F_\ell(F_1(p^\prime,r^\prime),q^\prime)\label{gen-eq3},\\
F_{\ell}(F_1(p,q),c)&=F_{\ell}(F_1(p^\prime,q^\prime),c^\prime).\label{gen-eq4} \end{align}

But $F_k$ and $F_\ell$ are orthogonal Latin squares, hence Equations \eqref{gen-eq1} and \eqref{gen-eq3}  imply
$F_1(p,r)=F_1(p^\prime,r^\prime)$ and $q=q^\prime$, while  Equations
\eqref{gen-eq2} and \eqref{gen-eq4} imply $F_1(p,q)=F_1(p^\prime,q^\prime)$ and $c=c^\prime$. Thus we may deduce that
$p=p^\prime$ and hence $r=r^\prime$. So $((p,r),
(q,c))=((p^\prime,r^\prime), (q^\prime,c^\prime))$, a contradiction. Hence $\{{\cal X}_k \mid 1\leq k \leq t\}$, is a
set of $t$ MOLS($n^2$).

Finally  assume that for some $k\in\{1,\dots,t\}$, ${\cal X}_k$ and  ${\cal B}$ are not orthogonal. Thus there exist
distinct cells $((p,r), (q,c))$ and $((p^\prime,r^\prime),
(q^\prime,c^\prime))$ such that \begin{align*}
(F_k(F_1(p,r),q),F_k(F_1(p,q),c))&=(F_k(F_1(p^\prime,r^\prime),q^\prime),F_k(F_1(p^\prime,q^\prime),c^\prime))\mbox{
and}\\
(F_1(p,q),L(F_1(p,r),c))&=(F_1(p^\prime,q^\prime),L(F_1(p^\prime,r^\prime),c^\prime)). \end{align*} Then \begin{align}
F_k(F_1(p,r),q)&=F_k(F_1(p^\prime,r^\prime),q^\prime),\label{gen-eq5}\\
F_k(F_1(p,q),c)&=F_k(F_1(p^\prime,q^\prime),c^\prime),\label{gen-eq6}\\
F_1(p,q)&=F_1(p^\prime,q^\prime)\label{gen-eq7},\\ L(F_1(p,r),c)&=L(F_1(p^\prime,r^\prime),c^\prime).\label{gen-eq8}
\end{align} Since $F_k$ is a Latin square, Equation \eqref{gen-eq7}  substituted into Equation \eqref{gen-eq6} gives
$c=c^\prime$. Then Equation \eqref{gen-eq8} gives
$F_1(p,r)=F_1(p^\prime,r^\prime)$ and when substituted into Equation \eqref{gen-eq5} gives $q=q^\prime$. Returning to
Equation \eqref{gen-eq7} we get $p=p^\prime$ and  consequently
$r=r^\prime$. So $((p,r), (q,c))=((p^\prime,r^\prime), (q^\prime,c^\prime))$, a contradiction. Hence for all $1\leq
k\leq t$, ${\cal X}_k$ is orthogonal to ${\cal B}$, and the result follows.
\end{proof}

\begin{corollary}
 Let $P$ be a partial Latin square of order $n$, $n\geq 3$. Then $P$ can be embedded in a Latin square ${\cal B}$ of
order at most $16n^2$,  where  ${\cal B}$ has at least $2n$ mutually orthogonal mates. Furthermore if $P$ is
idempotent then ${\cal B}$ can be constructed  to be idempotent.

\end{corollary}

\begin{proof} We will first embed $P$ in a Latin square $L$ of order $m$ where $2^k=m> 2n \geq 2^{k-1}$ which is always
possible given Evans result, Theorem \ref{Evans}. We can also assume that $L$ is indexed by
$[m]=\{0,1,\ldots,m-1\}$. As is well known, since $m$ is a prime power, there exists a set of $m-1$ mutually orthogonal
Latin squares $\{F_1,F_2,\ldots,F_{m-1}\}$ of order $m$, also indexed by $[m]$ and in standard form (that is,
$F_i(0,j)=j$ for each $1\leq i\leq m-1$ and $0\leq j\leq m-1$). Then $\{{\cal X}_1, {\cal X}_2,\ldots,{\cal
X}_{m-1},{\cal B}\}$ forms a set of size $m$ of mutually orthogonal Latin squares of order $m^2$.

Observe that since $F_1(0,r)=r$, the construction places a copy of $P$ in the sub-array defined by $p=0$ and $q=0$ and
so $P$ has been embedded in $\cal B$ which has  been shown to have $m-1$ mutually orthogonal mates.

As $2^k=m> 2n \geq 2^{k-1}$ we have  $2^{k+1}> 4n \geq 2^k=m$, so $16n^2\geq m^2$. Hence every partial Latin square of order
$n$ embeds in a Latin square of order at most $16n^2$  for which there exists at least $2n$ mutually orthogonal mates.

Now, one can make sure ${\cal B}$ is idempotent if $P$ is idempotent. When embedding $P$, ensure that $L$ is
idempotent, which can be guaranteed by Theorem \ref{IdeEm} because $m\ge 2n+1$. Note that $F_1$ is in
standard form and is decomposable into transversals. So there exists a transversal of $F_1$ involving the element
$(0,0,0)$. Without loss of generality one can assume that this transversal is
on the main diagonal of $F_1$. So $F_1(p,p)\neq F_1(p',p')$ for $p\neq p'$. Hence, if $p\ne p'$, the cells
$((p,r),(p,r))$ and $((p',r),(p',r))$ of $\mathcal{B}$ contain elements with
different first coordinates. The second coordinate in cell $((p,r),(p,r))$ of $\mathcal{B}$ is $L(F_1(p,r),r)$. So for
each fixed $p$, these second coordinates form a row-permuted copy of
$L$.

Now consider the subsquare $\mathcal{S}_p$ of $\mathcal{B}$ formed by the cells $((p,r),(p,r'))$ for $0\le r,r'\le
m-1$. The entries in $\mathcal{S}_p$ all have the same first coordinate
$F_1(p,p)$, and the second coordinates form a row-permuted copy of $L$. Since $L$ is idempotent, $L$ has a transversal
and by permuting the rows $\{(p,0),(p,1),\ldots,(p,m-1)\}$ of $\cal B$  we
can arrange for this transversal of $\mathcal{S}_p$ to lie on the main diagonal of $\mathcal{B}$. This can be done
independently for each $p=0,1,\ldots,m-1$, and the result is a transversal
of $\mathcal{B}$ on its main diagonal.  By suitable renaming of the elements of $\mathcal{B}$ we can then arrange for
$\mathcal{B}$ to be idempotent. In the case $p=0$, the original entry in
the cell $(0,r),(0,r')$ of $\mathcal{B}$ is $(0,L(r,r'))$, so no permuting of the rows of $\mathcal{S}_0$ or renaming
of elements $(0,x)$ is required (strictly speaking we apply the identity
permutation and the identity renaming here). Hence $\mathcal{B}$ retains a copy of $L$ in the subsquare
$\mathcal{S}_0$. Finally, to complete the proof, we apply the same permutation of the rows and renaming of elements to
each $\mathcal{X}_k$ as were applied to $\mathcal{B}$. \end{proof}

Note that one can increase the number of mutually orthogonal Latin squares that are orthogonal to ${\cal B}$ as much as
one likes by increasing the order of the embedding Latin square $L$ to
guarantee the existence of a larger number of mutually orthogonal Latin squares of the same order as $L$.

\begin{corollary}
 Let $L$ be a Latin square of order $n$ with $n\geq 7$ and $n\neq 10, 18\ or\ 22$. Then $L$ can be embedded in a Latin
square $\cal
B$ of order $n^2$ where ${\cal B}$ has at least four mutually orthogonal mates. \end{corollary}

\begin{proof}
 We know by \cite{Handbook} (Section III-3-4) and \cite{Todorov} that if $n\geq 7$ and $n\neq 10,18\ or\ 22$, there
exist four mutually orthogonal Latin squares of order $n$. Use these Latin squares to form $\cal B$, ${\cal X}_1$,
${\cal X}_2$, ${\cal X}_3$ and ${\cal X}_4$. \end{proof}

A {\em bachelor Latin square} is a Latin square which has no orthogonal mate; equivalently, it is a Latin square with
no decomposition into disjoint transversals. A {\em confirmed bachelor
Latin square} is a Latin square that contains an entry through which no transversal passes.

Wanless and Webb \cite{Wanless} have established the existence of confirmed bachelor Latin squares for all possible
orders $n$, $n\notin \{1, 3\}$. So it is interesting to note that the above
results (including Jenkins result) established that when one essentially ``squares'' a bachelor, it is possible to find
an orthogonal mate.

\section{Embedding a pair of OPLS in a set of MOLS} In this section we make use of the embedding result of Donovan and
Yaz{\i}c{\i}, \cite{Donovan}, to show that a pair of orthogonal partial
Latin squares can be embedded in pair of orthogonal Latin square which have many orthogonal mates.

\begin{theorem}[\cite{Donovan}] \label{D&Y} Let $P$ and $Q$ be a pair of orthogonal partial Latin squares of order $n$.
Then $P$ and $Q$ can be embedded in orthogonal Latin squares of order
$k^4$ and any order greater than or equal to $3k^4$ where $2^a=k\geq2n>2^{a-1}$ for some integer $a$. \end{theorem}

\begin{theorem}\label{many-squares} Let $A_1=[A_1(i,j)]$, $A_2=[A_2(i,j)]$ and $B_1=[B_1(i,j)]$, $B_2=[B_2(i,j)]$ be a
pair of orthogonal Latin squares of order $n$. Let $C_1=[C_1(i,j)],
\dots, C_t=[C_t(i,j)]$ be $t$ mutually orthogonal Latin squares of order $n$. Then the  squares \begin{align*} {\cal
B}_1&=\{((p,r),(q,c),(A_1(p,q),B_1(r,c)))\},\\ {\cal
B}_2&=\{((p,r),(q,c),(A_2(p,q),B_2(r,c)))\},\\ {\cal
X}_{i,f(i)}&=\{((p,r),(q,c),(C_{i}(p,B_1(r,c)),C_{f(i)}(q,B_2(r,c)))\}, \end{align*} where $i\in[t]=\{0,\ldots,t-1\}$
and $f:[t]\rightarrow [t]$ is a bijection,
form a  set of $t+2$ mutually orthogonal Latin squares of order $n^2$. \end{theorem}

\begin{proof} The arrays ${\cal B}_1$ and ${\cal B}_2$ may be obtained by taking direct products, so it is clear that
they are orthogonal Latin squares.

Assume that the array ${\cal X}_{\alpha,\beta}$ is not a Latin square, for some $\alpha,\beta$. Then there exists
$(p,r)$ such that
$(C_{\alpha}(p,B_1(r,c)),C_{\beta}(q,B_2(r,c))=(C_{\alpha}(p,B_1(r,c')),C_{\beta}(q',B_2(r,c'))$, for some
$(q,c),(q',c')$ with $(q,c)\ne (q',c')$, or there exists $(q,c)$ such that
$(C_{\alpha}(p,B_1(r,c)),C_{\beta}(q,B_2(r,c))$ $=(C_{\alpha}(p',B_1(r',c)),C_{\beta}(q,B_2(r',c))$, for some
$(p,r),(p',r')$ with $(p,r)\ne (p',r')$. The former case implies \begin{align}
C_{\alpha}(p,B_1(r,c))&=C_{\alpha}(p,B_1(r,c')),\label{many-sq1}\\
C_{\beta}(q,B_2(r,c))&=C_{\beta}(q',B_2(r,c')).\label{many-sq2} \end{align} By \eqref{many-sq1} $c=c'$ and so
\eqref{many-sq2} implies $q=q'$, a contradiction. The latter case implies \begin{align}
C_{\alpha}(p,B_1(r,c))&=C_{\alpha}(p',B_1(r',c)),\label{many-sq3}\\
C_{\beta}(q,B_2(r,c))&=C_{\beta}(q,B_2(r',c)).\label{many-sq4} \end{align} But then \eqref{many-sq4} implies $r=r'$ and
by \eqref{many-sq3} $p=p'$, a contradiction. Hence ${\cal
X}_{\alpha,\beta}$ is a Latin square.

Next take distinct $\alpha$ and $\gamma$, and consequently distinct $\beta$ and $\delta$, where $\beta=f(\alpha)$ and
$\delta=f(\gamma)$. Then assume that for distinct cells $((p,r),(q,c))$
and $((p',r'),(q',c'))$ \begin{align*}
(C_{\alpha}(p,B_1(r,c)),C_{\beta}(q,B_2(r,c)))&=(C_{\alpha}(p',B_1(r',c')),C_{\beta}(q',B_2(r',c'))),\\
(C_{\gamma}(p,B_1(r,c)),C_{\delta}(q,B_2(r,c)))&=(C_{\gamma}(p',B_1(r',c')),C_{\delta}(q',B_2(r',c'))). \end{align*}
Then \begin{align}
C_{\alpha}(p,B_1(r,c))&=C_{\alpha}(p',B_1(r',c')),\label{many-sq15}\\
C_{\beta}(q,B_2(r,c))&=C_{\beta}(q',B_2(r',c')),\label{many-sq16}\\
C_{\gamma}(p,B_1(r,c))&=C_{\gamma}(p',B_1(r',c')),\label{many-sq17}\\
C_{\delta}(q,B_2(r,c))&=C_{\delta}(q',B_2(r',c')).\label{many-sq18} \end{align} But $C_{\alpha}$ is orthogonal to
$C_{\gamma}$  and so Equations \eqref{many-sq15} and \eqref{many-sq17} imply $p=p'$ and $B_1(r,c)=B_1(r',c')$. Further
$C_{\beta}$ is orthogonal to $C_{\delta}$ and so Equations
\eqref{many-sq16} and \eqref{many-sq18} imply $q=q'$ and $B_2(r,c)=B_2(r',c')$. Finally $B_1$ and $B_2$ are orthogonal
and so $r=r'$ and $c=c'$. But this contradicts the assumption that the
cells $((p,r),(q,c))$ and $((p',r'),(q',c'))$ are distinct. Hence ${\cal X}_{\alpha,\beta}$ and  ${\cal
X}_{\gamma,\delta}$ are orthogonal.

Finally we prove that ${\cal B}_1$ and ${\cal X}_{\alpha,\beta}$ are orthogonal. Assume this is not the case and that
there exist distinct cells $((p,r),(q,c))$ and  $((p',r'),(q',c'))$ such
that \begin{align*} (A_1(p,q),B_1(r,c))&=(A_1(p',q'),B_1(r',c')),\\
(C_\alpha(p,B_1(r,c)),C_\beta(q,B_2(r,c)))&=(C_\alpha(p',B_1(r',c')),C_\beta(q',B_2(r',c'))). \end{align*} Then
\begin{align} A_1(p,q)&=A_1(p',q'),\label{meq1}\\ B_1(r,c)&=B_1(r',c'),\label{meq2}\\
C_\alpha(p,B_1(r,c))&=C_\alpha(p',B_1(r',c')),\label{meq3}\\
C_\beta(q,B_2(r,c))&=C_\beta(q',B_2(r',c')).\label{meq4} \end{align}

Since $C_\alpha$ is a Latin square substituting Equation \eqref{meq2} into Equation \eqref{meq3} implies $p=p'$. Now
since $A_1$ is a Latin square Equation \eqref{meq1} implies $q=q'$. Then,
since $C_\beta$ is a Latin square, Equation \eqref{meq4} implies  $B_2(r,c)=B_2(r',c')$. But $B_1$ and $B_2$ are
orthogonal so Equation \eqref{meq2} then gives $r=r'$ and $c=c'$. Consequently
${\cal B}_1$ and ${\cal X}_{\alpha,\beta}$ are orthogonal. Similarly we can show  ${\cal B}_2$ and ${\cal
X}_{\alpha,\beta}$ are orthogonal.

\end{proof}

\begin{corollary} For any $t\ge 2$, a pair of mutually orthogonal partial Latin squares of order $n$ can be embedded in
a set of $t$ mutually orthogonal Latin squares of polynomial
order with respect to $n$. \end{corollary}

\begin{proof} Let $A_1$ and $A_2$ be two orthogonal partial Latin squares of order $n$. By \cite{Donovan} we can embed
them into two orthogonal Latin squares ${\cal A}_1$ and ${\cal A}_2$ of
order $k^4$ where $2^a=k\geq2n>2^{a-1}$. As $k$ is a power of a prime, there are at least $k^4-1$ MOLS$(k^4)$. So there
are at least $(k^4-1+2)$ MOLS$(k^8)$ two of which contains the copies of
${\cal A}_1$ and ${\cal A}_2$. Similarly by choosing the order of ${\cal A}_1$ and ${\cal A}_2$ larger, one can obtain
as many orthogonal mates as one wants at the expense of increasing the
order of the squares into which the partial Latin squares are embedded. \end{proof}

Obviously Theorem \ref{many-squares} can also be used to construct mutually orthogonal Latin squares of order $n^2$ for
a given integer $n$. For example, in the literature only $8$ mutually
orthogonal Latin squares of order $576$ were know to exist, but the following corollary constructs $9$ MOLS$(576)$.

\begin{corollary} There are $9$ mutually orthogonal Latin Squares of order $576$. \end{corollary}

\begin{proof} By \cite{Handbook} Table 3.87 there are at least $7$ mutually orthogonal Latin squares of order $24$.
When applied in the construction given in Theorem \ref{many-squares}, we
may obtain $7+2=9$ mutually orthogonal Latin squares of order $24^2=576$. \end{proof}

\end{document}